\documentclass{amsart}
\usepackage{amsmath,amssymb,amsfonts,amsthm}
\usepackage{amssymb}
\usepackage{pgf}
\usepackage{todonotes}
\usepackage{epsfig}
\usepackage{mathrsfs}
\usepackage{verbatim}
\newtheorem{theorem}{Theorem}[section]

\newtheorem{proposition}[theorem]{Proposition}

\theoremstyle{definition}

\theoremstyle{remark}

\numberwithin{equation}{section}

\newcommand{\R}{\mathbb{R}}
\newcommand{\N}{\mathbb{N}}

\newcommand{\Q}{\mathbb{Q}}
\newcommand{\X}{\mathrm{X}}

\newcommand{\Y}{\mathrm{Y}}

\renewcommand{\S}{\mathbf{S}}

\newcommand{\esssup}{\mathop{\rm ess\ sup}}

\begin{document}

\title[On isometric embeddings of some function spaces]{Note on order-isomorphic isometric embeddings of some recent function spaces}

\author{Jarno Talponen}
\address{University of Eastern Finland, Department of Physics and Mathematics, P.O.Box 111, FI-80101 Joensuu, Finland}
\email{talponen@iki.fi}

\date{\today}

\begin{abstract}
We investigate certain recently introduced ODE-determined varying exponent $L ^p$ spaces. It turns out that these spaces are finitely representable in a concrete universal varying exponent $\ell^p$ space. Moreover, this can be accomplished in a natural unified fashion. 
This leads to order-isomorphic isometric embeddings of all of the above $L^p$ spaces to an ultrapower of the above varying exponent $\ell^p$ space.
\end{abstract}

\maketitle

\section{Introduction}

In this note we study the local theory of some very recently introduced varying-exponent $L^p$ spaces.

It is well-known that the classical $L^p$ spaces are finitely representable in the respective $\ell^p$ space. Moreover, the relevant finite-dimensional isomorphisms witnessing the finite representability can be chosen in such a way that they preserve bands, very roughly speaking. For example, the finite-representability of Bochner $L^p (L^q)$ spaces in the corresponding double $\ell^p$ spaces has been recently studied, see \cite{turk} (cf. \cite{Le}, \cite{L}). Recall that there is a known connection between finite representability and ultraproducts, and, in fact, 
in the above mentioned paper the local theory of these spaces is investigated by means of ultraproducts.

It is reasonable to ask if an analogous finite representability result holds in the varying exponent case, i.e. for spaces $L^{p(\cdot)}$. 
Here we show that it \emph{does} for a pair of recent classes of varying-exponent $L^{p(\cdot)}$ and $\ell^{p(\cdot)}$ spaces. We also prove some results involving ultraproducts but it turns out that there are versions of these results which are actually not specific to the ultrapower methods. 

The varying-exponent $\ell^{p(\cdot)}$ space investigated here, with exponents 
\[p(1)=p_1,\ p(2)=p_2\ , \ldots \in [1,\infty],\] 
can be described naively as follows:
\[\ell^{p(\cdot)} = (\ldots (((\R \oplus_{p_1} \R) \oplus_{p_2} \R)\oplus_{p_3} \R)\oplus_{p_4} \ldots .\]
Here $\R$ denotes a $1$-dimensional Banach space and the construction of the above space in \cite{talponen_y} is rigorous.
  
There is a natural `continuous version' of the above space. The author introduced in 
\cite{talponen_x} a class a varying-exponent $L^p$ spaces whose norm $\|f\|$ is governed by an ordinary d
differential equation as follows: 
\begin{equation}\label{eq: mainDE}
\varphi_f (0)=0 ,\ \varphi_{f}' (t)=\frac{|f(t)|^{p(t)}}{p(t)}\varphi_{f}(t)^{1-p(t)}\quad \mathrm{for\ a.e.}\ t\in [0,1].
\end{equation}
Here $p\colon [0,1] \to [1,\infty)$ and $f\colon [0,1]\to \R$ are measurable functions and $\varphi_f$ is Carath\'eodory's weak solution which exists and is unique for an initial value $\varphi_f (0)=0^+$ (see the paper for details). In the case with $\esssup p <\infty$ the set 
\[\{f \in L^0 \colon \varphi_f \ \mathrm{solution\ exists},\ \varphi_f (1)<\infty\}\]
becomes a Banach lattice with the usual point-wise operations defined a.e. and the norm 
\[\|f\|:=\varphi_f (1) .\]
In the constant $p(\cdot)=p\in [1,\infty)$ case this construction reproduces the classical $L^p$ spaces.

This paper also illustrates the inner workings of the above recent classes of spaces.

\subsection{Preliminaries}

We refer to the monographs in the references and the survey \cite{Heinrich} for a suitable background information. 
Throughout we are assuming the familiarity with the papers \cite{talponen_y} and \cite{talponen_x} 
regarding the construction, notations and basic facts involving $\ell^{p(\cdot)}$ and $L^{p(\cdot)}$ spaces, respectively.

If $p\colon [0,1]\to [1,\infty)$ is any measurable function there is a natural Banach function space $L^{p(\cdot)}_0$ such that 
$p(\cdot)$ is, intuitively speaking, almost bounded on this space. The space can be defined as the completion
\[\overline{\bigcup_{\alpha<\infty} \{1_{p(t)<\alpha} f\colon f \in L^{p(\cdot)}\} } \subset L^{p(\cdot)}.\]
(The space on the right can be regarded as a metric space but it may be non-linear for some cases of $p(\cdot)$.)

The double varying-exponent $\ell^{p(\cdot)}$ spaces, i.e. $\ell^{p(\cdot)} (\ell^{s(\cdot)} )$ can be defined as follows. 
For infinite matrices $(x_{n,m})_{n,m\in \N} \subset \R^{\N\times \N}$ we define the values of the corresponding norms in $2$ phases.
First, we let $a_k  := \|x_{k,\cdot }\|_{\ell^{s(\cdot)}}$ for all $k\in\N$. Then we set 
\[\|(x_{n,m})\|_{\ell^{p(\cdot)} (\ell^{s(\cdot)} )} := \|(a_k )\|_{\ell^{p(\cdot)}}.\]
In both phases we exclude the matrices $(x_{n,m})$ producing infinite values. It is easy to see that this results in a Banach space 
and it is denoted by $\ell^{p(\cdot)} (\ell^{s(\cdot)} )$.

For $a,b\geq 0$ and $1\leq p < \infty$ we denote
\[a \boxplus_p b =(a^p + b^p )^{\frac{1}{p}}.\]

If $\mathcal{F}$ is a filter on $\N$ and $a_n ,a \in \R$, $n\in\N$, we denote by $\lim_{n,\mathcal{F}} a_n =a$ the fact that 
\[\forall \varepsilon >0 \colon \{n\in \N\colon |a_n - a| < \varepsilon\}\in \mathcal{F}.\]

Recall that a Banach space $\X$ is finitely representable in a Banach space $\Y$ if for each finite-dimensional subspace 
$E \subset \X$ and $\varepsilon>0$ there is a finite-dimensional subspace $F \subset \Y$ and a linear isomorphism 
$T\colon E \to F$ with $\|T\| \|T^{-1}\| < 1+ \varepsilon$.
 
Given a Banach space $\X$ we denote 
\[\mathcal{C}(\X)=\ell^\infty (\X) / c_0 (\X), \]
adopting the notation used for Calkin algebras.

\section{Results}

\subsection{Preparations: Banach lattices of ultraproducts}

Let $\X_n$ be a sequence of Banach lattices, each satisfying the property
\begin{equation}\label{eq: left_triangle}
||x|-|y||\leq |x-y|.
\end{equation}
Observe that this condition immediately guarantees that the absolute value mapping $x\mapsto |x|$ is non-expansive. 
We let 
\[\bigoplus_{n}  \X_n \quad (\ell^\infty \ \mathrm{sense})\]
be the $\ell^\infty$ direct sum of the spaces. Write $\X = \bigoplus_{n} \X_n $.
Suppose that $\mathcal{F}$ is a filter on $\N$, e.g. a Fr\'echet filter. Then 
we let 
\[N_\mathcal{F} =\{(x_n ) \in \X\colon \lim_{n,\mathcal{F}} \|x_n\|_{\X_n} =0\}\subset \X .\]

It is easy to check that this is a closed subspace. For example, if $\X_n = \R$ for all $n$, a $1$-dimensional Banach space, 
and $\mathcal{F}$ is the filter generated by cofinite subsets, then $ N_\mathcal{F} =c_0$.  

We may generate a vector lattice order $\preceq$ on the space $\X / N_\mathcal{F} $ from the condition
\[ \{n\colon x_n \leq_{\X_n } y_n \}\in \mathcal{F} \quad \Rightarrow \quad (x_n ) /\sim\  \preceq\  (y_n )/\sim  .\]
(We are not claiming reverse implication above. This is so for instance because the equivalence classes 
do not determine the corresponding sequences uniquely.)

An alternative approach is that we may define an absolute value $|\cdot |$ on $\X / N_\mathcal{F} $ by 
\[(x_n ) /\sim\ \mapsto (|x_n|_{\X_n})/\sim .\]

Indeed, this is well defined since the absolute values $|\cdot|_{\X_n}$ satisfy \eqref{eq: left_triangle}.
Then the condition $| x | =x$ characterizes a positive cone which can be used in recovering the order $\preceq$. 
It is not hard to verify that these separate constructions result in the vector lattice order.

\begin{proposition}
Let us retain the above notations and assume that the absolute values $|\cdot|_{\X_n}$ satisfy \eqref{eq: left_triangle}, respectively. 
Then $\X$ endowed with the partial order $\preceq$ is a Banach lattice whose absolute value coincides with the mapping $|\cdot|$.  
\end{proposition}
\qed

We denote by $\mathcal{U}$ a free ultrafilter on the natural numbers.
Recall that the ultrapower of a Banach space $\Y$ is defined as 
\[\Y^\mathcal{U} = \ell^\infty (\Y) / N_\mathcal{U}.\]

\subsection{Order isomorphic isometric embeddings}

Let $r\colon \N \to \Q \cap [1,\infty)$ be a bijection, i.e. an enumeration of the rationals $q\geq 1$. 
Denote by $\ell^{r(\cdot)}_0 = [(e_n)] \subset \ell^{r(\cdot)}$. It is known that $\ell^{r(\cdot)}$ (resp. $\ell^{r(\cdot)}_0$) contains almost isometrically 
all the spaces of the type $\ell^{q(\cdot)}$ (resp. $\ell^{q(\cdot)}_0$), in particular the spaces $\ell^p$ (resp. $\ell^p$ for $p<\infty$), see \cite{talponen_y}.

\begin{theorem}
Let $r(\cdot)$ be as above.
The space $\mathcal{C}(\ell^{r(\cdot)}_0 )$ is universal for spaces of the type $L^{p(\cdot)}_0 [0,1]$. More precisely, the latter spaces considered with their a.e. point-wise order, can be mapped by a linear order-preserving isometry into $\mathcal{C} (\ell^{r(\cdot)}_0 ) $, endowed with the Banach lattice order $\preceq$, as described above. Moreover, the same conclusion holds if we consider the ultrapower $(\ell^{r(\cdot)}_0 )^\mathcal{U}$ in place of 
$\mathcal{C} (\ell^{r(\cdot)}_0 ) $ for any free ultrafilter $\mathcal{U}$ on $\N$.
\end{theorem}

\newcommand{\E}{\mathbb{E}}
\begin{proof}
We prove the latter statement involving the ultrapower which is more abstract (if not more complicated). Let $p\colon [0,1]\to [1,\infty)$ 
be a measurable function and $\mathcal{U}$ a free ultrafilter on $\N$. 

Note that $Y \oplus_p L^p (\mu)$ can be written isometrically as $(Y \oplus_p L^p (\mu_1 ) )  \oplus_p L^p (\mu_2 )$ where $\mu = \mu_1 + \mu_2$ is a decomposition such that $\max \mathrm{supp}\ \mu_1 \leq \min \mathrm{supp}\ \mu_2$.

According to Lusin's theorem there is a sequence of compact sets $C_n \subset [0,1]$ such that $p|_{C_n}$ are uniformly continuous for each $n$ and 
$m(C_n ) \to 1$. Since the norm-defining solutions $\varphi_{f}$ are assumed to be absolutely continuous and taking into account the basic properties of the solutions (see \cite{talponen_x}), we may identify 
\[\|f\|_{L^{p(\cdot)}} =\sup_n \|1_{C_n} f\|_{L^{p(\cdot)}},\quad f\in L^{p(\cdot)}_0 .\]
Indeed, let us recall the justification for this. It was proved in \cite{talponen_x} that $\|\cdot \|_{L^{p(\cdot)}}$ is a lattice norm and moreover that $\varphi_f \leq \varphi_g$ point-wise if $|f|\leq |g|$ point-wise a.e.
Since $\esssup p(\cdot) <\infty$, it is known that since $f \in L^{p(\cdot)}$ then also 
$1_{C_n} f \in L^{p(\cdot)}$, and $\varphi_{1_{C_n} f} \leq \varphi_{f}$ point-wise, see \cite{talponen_y}. 
Then, inspecting the governing differential equation \eqref{eq: mainDE}, we get immediately that 
\begin{equation}\label{eq: major}
\varphi_{1_{C_n} f}' \geq \varphi_{f}' 
\end{equation}
a.e. on $C_n$ and of course $\varphi_{1_{C_n} f}' =0$ a.e. in the 
complement of $C_n$. On the other hand, the solution $\varphi_f$, by its definition, is absolutely continuous
which implies 
\[\int_{[0,1]\setminus C_n} \varphi_{f}' (t)\ dt \to 0\]
as $n\to\infty$. Thus, using \eqref{eq: major} we get  
\[\sup_{n} \|1_{C_n} f \|_{L^{p(\cdot)}} = \sup_{n} \int_{[0,1]} \varphi_{1_{C_n} f}' (t)\ dt \geq 
\int_{[0,1]} \varphi_{f} ' (t)\ dt = \|f \|_{L^{p(\cdot)}} .\]
Since $\|1_{C_n} f \|_{L^{p(\cdot)}} \leq  \|f \|_{L^{p(\cdot)}}$ for all $n \in \N$, we observe that the above 
inequality becomes equality.

{\noindent \bf Step 1: Approximation of the norm by simple seminorms.}
First we assume that $f \in L^\infty$. This makes sense because it was shown in \cite{talponen_x} that $L^\infty$ is dense in $L^{p(\cdot)}$ in the case where $\esssup p < \infty$.

Consider simple semi-norms (as in \cite{talponen_x}), 
\[|f|_{N} := |f|_{(L^{p_1}(\mu_1 ) \oplus_{q_2} L^{p_2}(\mu_2 ) )\oplus_{q_3} L^{p_3}(\mu_3 ) )
\oplus_{q_4} \ldots ) \oplus_{q_n} L^{p_n}(\mu_n ) }.\]

Here the measures $\mu_i \colon \Sigma_i \to [0,1]$ are obtained as restrictions of the Lebesgue measure to compact subsets $I_i \subset [0,1]$ where $\max I_i \leq \min I_{i+1}$. Thus 
$\Sigma_i = \{A \cap I_i \colon A\in \Sigma\}$ where $\Sigma$ is the $\sigma$-algebra of the completed
Lebesgue measure on the unit interval and $\mu_i (B)=m(B)$ for all $B \in \Sigma_i$.

Let $(N_n)$ be a sequence of  such semi-norms with $p_i \leq p(\cdot) \leq q_i$ on $\mathrm{supp}(\mu_i )$. Then by the construction of the $\|\cdot \|_{L^{p(\cdot)}}$ norm we have that
$|f|_{N_n} \leq \|f\|_{L^{p(\cdot)}}$ for each $n$ and $f\in L^{p(\cdot)}_0$. Indeed, this is due to the fact that $\|f\|_{L^{p(\cdot)}}$ is essentially defined as a supremum of such seminorms.

By a diagonal argument we may choose $N_n$ in such a way that 
\[\lim_{n\to\infty} N_n (1_{C_m} f) = \|1_{C_m} f\|_{L^{p(\cdot)}}\]
for each $f \in L^\infty$ and $m\in\N$.  Indeed, since $p$ is bounded and uniformly continuous on $C_m$ we may find for each $\varepsilon>0$ numbers $0=a_1 < a_2 < \ldots <a_j =1$
such that 
\begin{enumerate}
\item{The corresponding supports for measures satisfy $\mathrm{supp} (\mu_i ) \subset [a_i , a_{i+1}]$;} 
\item{$q_i = \sup_{C_m \cap [a_i , a_{i+1}]}  p(\cdot)$, $q_i = \sup_{C_m \cap [a_i , a_{i+1}]}  p(\cdot)$;}
\item{Intuitively, the differences $q_i - p_i$ are negligibly small;} 
\item{
\[\frac{d}{dt} |1_{C_m \cap [0,t]} f|_N \geq \varphi_f ` (t) - \varepsilon \]
a.e. on $C_m$ for $f$ such that $\sup_t |f(t)|=1$. }
\end{enumerate}
This is due to the fact that 
\begin{equation}\label{eq: converg}
\frac{d}{dt} \left(A^{q_i} + (t|x|^{p_i})^\frac{q_i}{p_i} \right)^\frac{1}{q_i} \bigg|_{t=0} \to \frac{|x|^{p}}{p} A^{1-p}
\end{equation}
uniformly for $0 \leq A, |x|\leq 1$ as $p_i \nearrow p$, $q_i \searrow p$. The diagonal argument is then applied to choose 
the sequence of seminorms $N_n$ as to eventually cover all cases $\varepsilon = \frac{1}{m}$, $C_m$ and $|f|\leq m$ for all $m\in\N$.  

For each $k\in\N$ let $\mathcal{F}_k$ be the finite $\sigma$-algebra generated by all the supports of 
$\mu_i$ corresponding to $N_n$ for $1\leq n \leq k$. Without loss of generality we may assume 
by adding suitable finitely many sets (e.g. dyadic decompositions of the unit interval) to each $\mathcal{F}_k$ that
\begin{equation}\label{eq:  diam}
\lim_{k\to\infty} \sup\{\mathrm{diam}(\Delta)\colon \Delta \in \mathcal{F}_k\ \mathrm{atom}\}=0.
\end{equation}
and that $\bigcup_k \mathcal{F}_k$ $\sigma$-generates the Borel $\sigma$-algebra on the unit interval.
By an atom of an algebra of sets $\mathcal{F}$ we mean $\Delta \in \mathcal{F}$  such that if $\Delta'  \in \mathcal{F}$, $\Delta' \subset \Delta$, then 
$\Delta' = \Delta$ or $\Delta' = \emptyset$.

Next we study the conditional expectation operators 
$\E (f \ |\  C, \mathcal{F}_k )$. Here 
\[\E (f\ |\ C, \mathcal{F}) = \sum_{\Delta } A_{\Delta\cap C}^* (f) 1_{\Delta\cap C}\] 
where $A_{\Delta\cap C}^* \in L^\infty$ is considered as the average integral (operator) over $\Delta \cap C$ where $\Delta\in  \mathcal{F}$
are atoms with respect to the finite algebra of sets $\mathcal{F}$. We use the convention that $A_{\Delta\cap C}^* (f)=0$ whenever $m(\Delta\cap C)=0$.

Restricting $f$ to the support of any of the $\mu_i$:s, it follows from the 
martingale convergence principle that $\E (f\ |\ C, \mathcal{F}_k ) \to 1_C f$ almost everywhere as $k\to \infty$ and 
also in the $L^{p_i}(\mu_i )$-sense, see e.g. \cite[Ch. 5.4.]{durret}. Consequently, putting the pieces together, we obtain that
\begin{equation}\label{eq: Econv}
|\E (f\ |\ C_m ,\mathcal{F}_k ) |_{N_n} \to |1_{C_m} f |_{N_n },\quad \mathrm{as}\ k\to\infty,\quad 
\forall n,m\in\N,\ f\in L^\infty .
\end{equation}
  
Define versions $N_{n}'$ of the semi-norms $N_n$ by replacing $q_i$ with $p_i$. 
By the uniform continuity of $p$ on the sets $C_m$, \eqref{eq: converg}, \eqref{eq:  diam} and \eqref{eq: Econv}
we may choose subsequences $(m_n ), (k_n) \subset \N$ with $m_n , k_n \to \infty$ as $n\to\infty$ such that
\[\lim_{n\to\infty} |\E (f | C_{m_n},\mathcal{F}_{k_n} ) |_{N_{n}'} = \|f\|_{L^{p(\cdot)}},\quad f\in \S_{L^{\infty}},\]
where $n\mapsto m_n$ need not be strictly increasing. In fact, the above equality clearly holds for any $f \in L^{\infty}$.

{\noindent \bf Step 2: Approximation of the required operator by tame non-linear operators.}

Consider a sequence $0\leq \alpha_n \nearrow \infty$ and non-linear operators $T_n \colon L^{p(\cdot)}_0 [0,1] \to L^{p(\cdot)}_0 [0,1]$ 
given by $T_n (f)[t]=\min(\alpha_n , \max(-\alpha_n , f(t)))$ for a.e. $t$. Thus the following condition holds:
\begin{enumerate}
\item[(a)] $\|f-T_n f\|_{L^{p(\cdot)}} \to 0$ as $n\to \infty$ for each $f\in L^{p(\cdot)}_0$.
\end{enumerate}
Indeed, $L^\infty \subset L^{p(\cdot)}$ is dense whenever $\esssup p <\infty$ and consequently 
it follows from the definition of the space $L^{p(\cdot)}_0$ that $L^\infty$ is dense in it as well.
We may additionally choose the above sequences of $\alpha$:s, conditional expectation operators and the semi-norms in such a way that 
\begin{enumerate}
\item[(b)] $\lim_{n\to\infty} \sup_{f\in L^{p(\cdot)}_0} |\ |\E(T_n (f) \ |\ C_{m_n} ,  \mathcal{F}_{k_n}) |_{N_{n}'} - \|1_{C_{m_n}} f \|_{L^{p(\cdot)}}\ | =0$.
\end{enumerate}
This can be established by using the uniform continuity of $p$ on the compact sets $C_m$, using the fact that in such a case the 
simple semi-norms converge uniformly and invoking the martingale $L^p$-convergence fact above.
Note that $T_n$:s are order-preserving although they are non-linear.

Next we analyze the simple semi-norms chosen and in particular the exponents $p_i$:s.  We obtain that for each $n$ for the exponents $p_{2}^{(n)}, p_{3}^{(n)}, \ldots , p_{j}^{(n)}$ corresponding to $N_{n}'$ there are 
$i_2 < i_3 < \ldots <i_j$ with 
rational exponents $r_{i_2}^{(n)}, r_{i_3}^{(n)} , \ldots , r_{i_j}^{(n)}$ very close to the corresponding $p$:s. Indeed, 
by repeating the almost isometric embedding construction in \cite{talponen_y} we may pick the $r_{i_k}$:s in such a way that 
\begin{equation}\label{eq: quasi}
\|\cdot\|_{\ell^{\left\{ r_{i_2}^{(n)}, r_{i_3}^{(n)} , \ldots , r_{i_j}^{(n)}\right\}}} \leq \|\cdot\|_{\ell^{\left\{ p_{2}^{(n)}, p_{3}^{(n)}, \ldots , p_{j}^{(n)}\right\}}}
\leq \frac{n+1}{n} \|\cdot\|_{\ell^{\left\{ r_{i_2}^{(n)}, r_{i_3}^{(n)} , \ldots , r_{i_j}^{(n)}\right\}}}.
\end{equation}
Here
\[\ell^{\left\{ q_{2}, q_{3}, \ldots , q_{j} \right\}} = (((\ldots (\R \oplus_{q_{2}} \R )\oplus_{q_3} \R)\oplus_{q_4} \ldots \oplus_{q_{j-1}} \R)\oplus_{q_j} \R ,\]  
that is, $\R^{j+1}$ with the norm 
\[\|(x_n )_{n=1}^{j+1} \|_{\ell^{\left\{ q_{2}, q_{3}, \ldots , q_{j}\right\}}}= (((\ldots (|x_1 | \boxplus_{p_2} |x_2 |) \boxplus_{p_3} |x_3|) \boxplus_{p_4} \ldots \ ) \boxplus_{p_j} |x_{j+1}|  .\]

We can find for each $n$ a finite-dimensional varying exponent $\ell^p$ space $\ell^{\left\{p_{2}^{(n)}, \ldots , p_{j}^{(n)}\right\}}$
and a natural linear order-preserving linear bijection 
\[B_n\colon \mathrm{Image}\ \E(\cdot \ |\ C_{m_n} ,  \mathcal{F}_{k_n}) \to  \ell^{\left\{p_{2}^{(n)}, \ldots , p_{j}^{(n)}\right\}},\] 
where $j$ is the finite dimension of the space of simple functions of the form $\E(f \ |\ C_{m_n} ,  \mathcal{F}_{k_n})$,
and such that 
\begin{equation}\label{eq: EfN}
|\E(f\ |\ C_{m_n} ,  \mathcal{F}_{k_n}) |_{N_{n}'} = \|B_n \E(f\ |\ C_{m_n} ,  \mathcal{F}_{k_n})\|_{\ell^{\left\{p_{2}^{(n)}, p_{j}^{(n)}\right\}}}, 
\end{equation}
$n\in\N$, $f\in L^{p(\cdot)}_0$.
Indeed, this applies the fact that $\mathcal{F}_{k_n}$ contains all the supports of $\mu_i$:s corresponding to $N_{n}'$ and we may write 
\[\E(f\ |\ C_{m_n} ,  \mathcal{F}_{k_n}) = \sum_{i=1}^j A_{\Delta_i}^* (f) 1_{\Delta_i} \]
and
\begin{multline*}
|\E(f\ |\ C_{m_n} ,  \mathcal{F}_{k_n}) |_{N_{n}'} = (\|A_{\Delta_1}^* (f) 1_{\Delta_1}\|_{L^{p_{1}^{(n)}}(\Delta_1 )} \boxplus_{p_{2}^{(n)}}  \|A_{\Delta_2}^* (f) 1_{\Delta_2}\|_{L^{p_{1}^{(n)}}(\Delta_2 )}) \boxplus_{p_{3}^{(n)}} \\ 
\ldots \ ) \boxplus_{p_{j}^{(n)}} \|A_{\Delta_j}^* (f) 1_{\Delta_j}\|_{L^{p_{j}^{(n)}}(\Delta_j )} 
\end{multline*}
where the subsets $\Delta_i \subset [0,1]$ are successive ($\sup \Delta_i \leq \inf  \Delta_{i+1}$) and are $\mathcal{F}_{k_n}$-atomic subsets of the supports of $\mu_i$:s corresponding to $N_{n}'$.
Recall that 
\[L^p (\Delta_i ) \oplus_p L^p (\Delta_{i+1})=L^p (\Delta_i \cup \Delta_{i+1})\] 
in a canonical way. The mapping $B_n$ is given by 
\[\E(f\ |\ C_{m_n} ,  \mathcal{F}_{k_n}) \mapsto \left(m(\Delta_1)^\frac{1}{p_1} A_{\Delta_1}^* (f),\ m(\Delta_2)^\frac{1}{p_2} A_{\Delta_2}^* (f),\ \ldots \ ,\ m(\Delta_j)^\frac{1}{p_j} A_{\Delta_j}^* (f) \right)\]
and it is easy to see that it is well-defined, linear and bijective.

Let $\iota_n \colon \ell^{\left\{ p_{2}^{(n)}, p_{3}^{(n)}, \ldots , p_{j}^{(n)}\right\}} \to \ell^{r(\cdot)}$ be a natural linear order-preserving mapping corresponding to the arrangement in \eqref{eq: quasi}. 
Next, we define a mapping $S\colon L^{p(\cdot)}_0 \to \ell^\infty (\ell^{r(\cdot)})$ as follows: $S(f)=(x_n)$ where 
\[x_n = \iota_n B_n \E(T_n (f)\ |\ C_{m_n} ,  \mathcal{F}_{k_n}).\]

Note that 
\begin{equation}\label{eq: Cconver}
\|1_{C_{m_n}} f \|_{L^{p(\cdot)}} \to \| f \|_{L^{p(\cdot)}},\quad n\to \infty 
\end{equation}
by the absolute continuity of the solutions $\varphi_f$, as observed above.

The mapping required in the statement is the induced mapping $\widehat{S}\colon f \mapsto S(f)+ N_\mathcal{F}$, mapping to the quotient space (in this case the ultrapower). This mapping is clearly order-preserving.

It is also norm-preserving, since $\lim_{n\to\infty} \|x_n\|_{  \ell^{r(\cdot)}}=\|f\|$. Indeed, this follows by using (b), \eqref{eq: quasi}, \eqref{eq: EfN}, \eqref{eq: Cconver} and the fact that 
$L^\infty$ is dense in $L_{0}^{p(\cdot)}$.

To verify the linearity of $\widehat{S}$, observe that for any $\varepsilon>0$ and $f,g \in L^{p(\cdot)}_0$ there are $f_0 , g_0 \in L^\infty$ such that $\max(\|f- f_0 \|_{L^{p(\cdot)}}, \|g- g_0 \|_{L^{p(\cdot)}})< \varepsilon$.
Thus, selecting $n\in \N$ in such a way that $\alpha_n \geq  \max (\|f\|_{L^\infty }, \|g\|_{L^\infty })$, we obtain that 
\[\max(\|f- T_n f \|_{L^{p(\cdot)}}, \|g- T_n g \|_{L^{p(\cdot)}}) \leq \max(\|f- f_0 \|_{L^{p(\cdot)}}, \|g- g_0 \|_{L^{p(\cdot)}})< \varepsilon .\]
Here we are using the fact that $L^{p(\cdot)}_0$ is a Banach lattice in its usual order (see \cite{talponen_x}). This means that 
\begin{multline*}
\|T_n (f+g) - (T_n f +T_n g)\|_{L^{p(\cdot)}}\leq \|T_n (f+g) - (f+g)\|_{L^{p(\cdot)}} \\
+ \|T_n f - f\|_{L^{p(\cdot)}} + \|T_n g +g\|_{L^{p(\cdot)}} \to 0
\end{multline*}
as $n\to\infty$. Recalling (b) and the construction of $S$, it follows that  
\[\|S_n (f+g) - (S_n (f) + S_n (g))\|_{\ell^{r(\cdot)}} \to 0,\quad n\to\infty.\]
This shows that $\widehat{S}(f+g)=\widehat{S}(f)+\widehat{S}(g)$ for all $f,g\in L^{p(\cdot)}_0$. The homogeneity of $\widehat{S}$ is seen similarly. This completes the proof.
\end{proof}

\begin{theorem}
Let $r(\cdot)$ be as above.
The space $\mathcal{C} (\ell^{r(\cdot)}(\ell^{r(\cdot)}))$ is universal for spaces of the type
$\ell^{q(\cdot)} (\ell^{s(\cdot)})$.  More precisely, the latter spaces considered with their matrix entry wise order can be mapped by a linear order-preserving isometry into 
$\mathcal{C} (\ell^{r(\cdot)}(\ell^{r(\cdot)}))$. Moreover, the same conclusion holds if we consider the ultrapower $(\ell^{r(\cdot)}(\ell^{r(\cdot)}) )^\mathcal{U}$ in place of 
$\mathcal{C} (\ell^{r(\cdot)}(\ell^{r(\cdot)}))$. 
In particular, each space $\ell^{q(\cdot)} (\ell^{s(\cdot)})$ is finitely representable in $\ell^{r(\cdot)}(\ell^{r(\cdot)})$.
\end{theorem}
\begin{proof}
(Sketch.) Consider each element of $\ell^{q(\cdot)} (\ell^{s(\cdot)})$ as a sequence $(x_n )$ with $x_ n\in \ell^{s(\cdot)}$ for $n\in\N$. We may consider these elements as infinite matrices $(x_{n,m})_{n,m\in\N}$. 

Let $k \in \N$. By repeating inductively the observation involving \eqref{eq: quasi} we can find $n_1 ,n_2 , \ldots , n_k$ and $m_1 ,m_2 , \ldots , m_k$ such that the mapping 
$\iota_k \colon (x_{i,j}) \mapsto (y_{n_i , m_j})$, and setting other coordinates $y_{n,m}$ to $0$, satisfies 
\[\|r_k (x_{i,j})\|_{\ell^{q(\cdot)} (\ell^{s(\cdot)})} \leq \| \iota_k r_k (x_{i,j})\|_{\ell^{r(\cdot)}(\ell^{r(\cdot)})} \leq \frac{k+1}{k} \|r_k (x_{i,j})\|_{\ell^{q(\cdot)} (\ell^{s(\cdot)})} .\]
Here the domain of $\iota_k$ is $\{(1_{n,m\leq k}\ x_{n,m})\colon  x_{n,m}\in\R,\ n,m\in\N\}$ and we denote by $r_k \colon \ell^{q(\cdot)} (\ell^{s(\cdot)}) \to \ell^{q(\cdot)} (\ell^{s(\cdot)})$ the canonical projection to this domain.

The required linear isometry is induced by the operator 
\[S \colon  \ell^{q(\cdot)} (\ell^{s(\cdot)}) \to \ell^\infty (\ell^{r(\cdot)}(\ell^{r(\cdot)})) ,\quad (x_{n,m}) \mapsto (\iota_k r_k (x_{n,m}))_{k\in\N} .\]

\end{proof}

We note that the previous result holds also as a left-handed version, where we consider all the varying-exponent $\ell^p$-spaces formally as 
\[\R \oplus_{p_1} (\R \oplus_{p_2} (\R \oplus_{p_3} (\ldots\quad\quad \ldots ))\ldots ).\]

\subsection*{Acknowledgments}

This work has been financially supported  by V\"ais\"al\"a foundation's and the Finnish Cultural Foundation's research grants and Academy of Finland Project \# 268009.

\end{document}